\documentclass{article}
\usepackage{tikz}
\usetikzlibrary{shapes,arrows,chains}
\pdfoutput=1
\usepackage[utf8]{inputenc}
\usepackage{array}
\usepackage{amsmath, amsthm, amssymb,xcolor}
\usepackage{enumerate, comment}
\usepackage{dsfont}
\usepackage{float}
\usepackage[colorlinks]{hyperref}
\hypersetup{
    colorlinks=true,
    linkcolor=blue,
    filecolor=magenta,
    urlcolor=gray,
    citecolor=gray,
}
\allowdisplaybreaks

\usepackage[export]{adjustbox}

\usepackage{todonotes}
\usepackage{titlesec}
\usepackage{multicol}
\usepackage{tikz}
\usepackage[capitalise]{cleveref}
\usepackage{standalone}
\usepackage{mathtools}
\usepackage{mathdots} 
\usepackage{adjustbox}
\usepackage[normalem]{ulem}
\usepackage{enumitem}
\usepackage{subcaption}
\usepackage[backend=biber, style=alphabetic, sorting=nyt, maxnames=100,backref=true]{biblatex}
\newtheorem{theorem}{Theorem}[section]
\newtheorem{lemma}[theorem]{Lemma}
\newtheorem{proposition}[theorem]{Proposition}
\newtheorem{corollary}[theorem]{Corollary}

\newtheorem{question}[theorem]{Question}
\newtheorem{observation}[theorem]{Observation}

\newtheorem*{theorem*}{Theorem}

\newtheorem{example}[theorem]{Example}

\newcommand{\ch}[1]{c_{H,#1}}
\newcommand{\cv}[1]{c_{\rm #1}}
\DeclareMathOperator{\diam}{diam}
\DeclareMathOperator{\rad}{rad}
\newcommand{\T}[0]{\mathcal{T}}

\newcommand*{\myproofname}{Proof}

\usepackage{color}

\usepackage{nicefrac}
\usepackage{setspace}

\title{$k$-Hyperopic Cops and Robber}

\begin{document}

\author{Nicholas Crawford \footnote{\href{mailto:nicholas.2.crawford@ucdenver.edu}{nicholas.2.crawford@ucdenver.edu}, Mathematical and Statistical Sciences, Universtiy of Colorado Denver, USA}
\and Vesna Ir\v si\v c Chenoweth \footnote{\href{mailto:vesna.irsic@fmf.uni-lj.si}{vesna.irsic@fmf.uni-lj.si}, Faculty of Mathematics and Physics, University of Ljubljana, Slovenia}}

\maketitle

\begin{abstract}
    A generalization of hyperopic cops and robber, analogous to the $k$-visibility cops and robber, is introduced in this paper. For a positive integer $k$ the $k$-hyperopic game of cops and robber is defined similarly as the usual cops and robber game, but with the robber being omniscient and invisible to the cops that are at distance at most $k$ away from the robber. The cops win the game if, after a finite number of rounds, a cop occupies the same vertex as robber. Otherwise, robber wins. The minimum number of cops needed to win the game on a graph $G$ is the $k$-hyperopic cop number $c_{H,k}(G)$ of $G$.

    In addition to basic properties of the new invariant, cop-win graphs are characterized and a general upper bound in terms of the matching number of the graph is given. The invariant is also studied on trees where the upper bounds mostly depend on the relation between $k$ and the diameter of the tree. It is also proven that the 2-hyperopic cop number of outerplanar graphs is at most 2 and an upper bound in terms of the number of vertices of the graph is presented for $k \geq 3$.
\end{abstract}


\noindent
{\bf AMS Subj.\ Class.\ (2020):} 05C57, 91A24, 05C70, 05C05, 05C10

\section{Introduction}
\label{sec:intro}

The game of Cops and Robber is a two-player game on a simple graph $G$. One player controls the cops, while the other player controls the robber. A round consists of a move of cops and a move of the robber. At the start of the game, the cops select their starting positions, and afterward the robber selects the starting position as well. Cops move first and then the players alternate taking moves. The cops and robber occupy vertices of the graph, and during a move, they can either stay on the same vertex or move to an adjacent one. The cops win the game if, after a finite number of moves, they catch the robber; that is, one of the cops occupies the same vertex as the robber. The robber wins if he can evade being captured indefinitely.  

The game of Cops and Robber was introduced independently by Quilliot \cite{Qui78} and Nowakowski and Winkler \cite{NoWi83} almost fifty years ago. Initially, only the game with one cop and one robber was considered, while the game with multiple cops was first studied in the 1980s by Aigner and Fromme \cite{AiFr84}. They considered the minimum number of cops needed to win the game on a graph $G$. This graph invariant is called the \emph{cop number} of $G$ and is denoted by $c(G)$. The game has received an astonishing amount of attention, leading also to several well-studied variations of the game; see for example the books \cite{BoNo11} and \cite{Bo22}.

One of the most investigated problems in the area is establishing good upper bounds for the cop number of connected graphs. It is known that the cop number of outerplanar graphs is at most 2 \cite{clarke-2002-phd}, that the cop number of planar graphs is at most 3, \cite{AiFr84} and that the cop number of toroidal graphs is also at most 3 \cite{Lehner2021}. However, the two main problems in the area are still open. Meyniel's conjecture \cite{Fr87} states that if $G$ is a connected graph of order $n$, then $c(G) = O(\sqrt{n})$, while the Schr{\"o}der's conjecture \cite{Schr01} claims that if $G$ is a graph of genus $g$, then $c(G) \leq g + 3$. Moreover, Mohar conjectured  \cite{Mo17} that if $G$ is a graph of genus $g$, then $c(G) = g^{\frac{1}{2} + o (1)}$.

The perfect information assumption in the game of cops and robber, i.e.\ the fact that both players have complete information about the location of the other player, is not a good model for several real-life applications of the game. It is more reasonable to assume that the robber's location is not always completely known to the cops. For example, in the localization game \cite{loc-game1, loc-game2}, the cops only know their respective distances to the robber, while in 0-visibility Cops and Robber \cite{tosic1986-0-vis, dereniowski+2015-0-vis, xue+2022-0-vis}, the robber is invisible unless one of the cops is on the same vertex. There are several similarities between 0-visibility Cops and Robber and graph searching \cite{breisch-1967-graph-searching, fomin+2008-graph-searching}, but in general the problems result in different games. 

For 0-visibility Cops and Robber, introduced in \cite{tosic1986-0-vis}, the location of the robber is unknown to cops unless one of them catches the robber, but apart from that, the rules are the same as in the usual cops and robber game on graphs. An important detail is that in this game we assume that the robber is \emph{omniscient}, meaning that the robber knows the complete strategy of the cops. As a result, the robber is never captured by chance. The minimum number of cops needed to capture the robber on a graph $G$ is the \emph{0-visibility cop number} $\cv{0}(G)$. For example, if the game is played on $K_3$, because the robber is omniscient, one cop cannot win. However, two cops can win, thus $c_0(K_3) = 2$. The game has been solved for complete graphs, complete bipartite graphs, paths and cycles \cite{tosic1986-0-vis}. In \cite{dereniowski+2015-0-visibility-2}, a construction of all trees with a given zero-visibility cop number is given, and it is known that the path-width of a graph is an upper bound for the 0-visibility cop number \cite{dereniowski+2015-0-vis}.

Afterwards, a variation of the game where the visibility of cops is limited to short distances has been introduced \cite{tang-2004-k-visibility, clarke+2020-k-visibility}. Let $k \geq 0$ be an integer. The \emph{$k$-visibility cops and robber} game is played analogous to the 0-visibility game, except that each cop sees the robber if the robber is at distance at most $k$ away from the cop. If $k=0$, this is equivalent to the 0-visibility game, and if $k \geq \diam(G)$, it is equivalent to the usual cops and robber game. The smallest number of cops needed to guarantee catching the robber on a graph $G$ is the \emph{$k$-visibility cop number} $c_k(G)$. Clearly, $c(G) \leq c_k(G) \leq c_j(G)$ for every $0 \leq j \leq k$ and every graph $G$. Among other results known about the game, we mention that in \cite{clarke+2020-k-visibility} trees with a given $k$-visibility cop number are characterized.

In 2019, Bonato, Clarke, Cox, Finbow, Inerney and Messinger introduced the variation of the game of cops and robber called the hyperopic cops and robber \cite{bonato-2019-hyperopic}, which is, in some sense, an inverse of the 1-visibility cops and robber. Here, cops are hyperopic or farsighted, not seeing the robber if they are close. The variation was motivated by real life application of emulating certain prey-predator systems (see for example \cite{janosov-2017}). In this game the robber is invisible if and only if they are adjacent to all cops. Again, we assume that robber is omniscient. The minimum number of cops required to win on a graph $G$ is the \emph{hyperopic cop number} $c_H(G)$. There are several known results about the game \cite{bonato-2019-hyperopic, clarke-2023-hyperopic}, including an upper bound for (outer)planar graphs, results for diameter 2 graphs and Cartesian products, and that connected graphs with hyperopic cop number 1 are exactly trees. 

In this paper, we introduce a generalization of hyperopic cops and robber, analogous to the $k$-visibility cops and robber. For a positive integer $k$ the \emph{$k$-hyperopic game of cops and robber} is defined similarly as the usual cops and robber game, but with the following modifications. While the robber is playing the perfect information game, the robber is invisible to the cops if for every cop it holds that $1 \leq d(r, c) \leq k$ where $r$ and $c$ are the current positions of the robber and the cop, respectively. The cops win the game if, after a finite number of rounds, a cop occupies the same vertex as robber. Otherwise, robber wins. We assume that robber is omniscient and knows all the strategies of the cops. The minimum number of cops needed to win the game on a graph $G$ is the \emph{$k$-hyperopic cop number} $\ch{k}(G)$ of $G$. Clearly, $\ch{1}(G) = c_H(G)$.
In the rest of the paper we consider only connected graphs. 

The paper is organized as follows. Basic definitions needed are defined at the end of this section. In Section \ref{sec:basic}, we present simple results related to different classes of graphs, establish the relation between $c(G)$, $c_0(G)$ and $c_{H,k}(G)$, and characterize graphs with $k$-hyperopic cop number equal to 1. In Section \ref{sec:general-upper}, we prove general upper bounds for $c_0(G)$ and $c_{H,k}(G)$. Section \ref{sec:trees} provides several results for the $k$-hyperopic cop number of trees and gives an upper bound in terms of $k$ in relation to the diameter of the tree. Finally in Section \ref{sec:outerplanar}, we prove that for an outerplanar graph $G$, $c_{H,2}(G) \leq 2$. We finish with a general upper bound on outerplanar graphs and possible directions for further research.

We begin by defining various subclasses of graphs. A \emph{caterpillar} is a tree such that if all leaves of the tree are removed, the remaining graph is a path. Since it will be used multiple times, we define the tree $\widehat{T}$ to be the subdivision of $K_{1,3}$ as shown in Figure \ref{fig:T-hat}. 

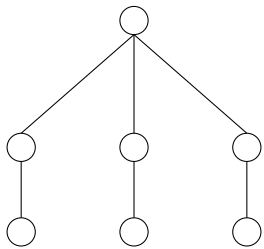
\begin{figure}
    \centering
\begin{tikzpicture} [scale=.75]
\tikzstyle{every node}=[font=\LARGE]
\draw (13,20.75) -- (11,19);
\draw (13,20.75) -- (13,19);
\draw (13,20.75) -- (15,19);
\draw (11,18.5) -- (11,17.5);
\draw (13,18.5) -- (13,17.5);
\draw (15,18.5) -- (15,17.5);
\draw (13,21) circle (0.25cm);
\draw  (11,18.75) circle (0.25cm);
\draw  (13,18.75) circle (0.25cm);
\draw  (15,18.75) circle (0.25cm);
\draw  (11,17.25) circle (0.25cm);
\draw  (13,17.25) circle (0.25cm);
\draw  (15,17.25) circle (0.25cm);
\end{tikzpicture}
    \caption{The tree $\widehat{T}.$}
    \label{fig:T-hat}
\end{figure}

We now move onto more standard definitions that have to do with the structure of our graph. A subset of the edges is a \emph{matching} if each vertex appears in at most one edge of that matching. We say $\alpha'(G)$ is the size of the maximum matching in $G$, called the \emph{matching number}. An \emph{induced subgraph}, $G[S]$, is a subgraph of $G$ formed from a subset of the vertices of the graph and all of the edges, from the original graph, connecting pairs of vertices in $S$. 

\section{Basic properties}
\label{sec:basic}

In this section we give basic upper and lower bounds for the $k$-hyperopic cop number, and exact values for several well-known graph classes.

\begin{observation}
    \label{obs:basic-bounds}
    If $G$ is a graph and $\ell \geq k \geq 1$, then
    $$c(G) \leq \ch{k}(G) \leq \ch{\ell}(G) \leq c_0(G).$$
\end{observation}

\begin{proof}
    A winning strategy for the $k$-hyperopic game performed in the usual cops and robber game is a winning strategy for the same number of cops, thus $\ch{k} \geq c(G)$. The other inequalities follow by an analogous argument.
\end{proof}

\begin{observation}
    \label{obs:0-vis}
    If $\diam(G) \leq k$, then $\ch{k}(G) = c_0(G)$.
\end{observation}

\begin{proof}
    Since $\diam(G) \leq k$, every vertex of $G$ is at distance at most $k$ from all other vertices. Thus robber is invisible during the whole game (unless he is caugh), so the $k$-hyperopic game is equivalent to the 0-visibility game.
\end{proof}

It is known that $c_0(K_{m,n}) = m$ if $n \geq m \geq 2$ \cite{tosic1986-0-vis}, while $\ch{1}(K_{m,n}) = 2$ if $n \geq m \geq 2$ \cite{bonato-2019-hyperopic}. Using Observation \ref{obs:0-vis} we obtain the following.

\begin{corollary}
    \label{cor:complete-bipartite}
    If $k \geq 2$ and $n \geq m \geq 1$, then $\ch{k}(K_{m,n}) = m$.
\end{corollary}

\begin{proposition}
\label{prop:graph-classes}
    Let $k \geq 1$. 
    \begin{enumerate}
        \item If $n \geq 1$, then $\ch{k}(P_n) = 1$.
        \item If $n \geq 3$, then $\ch{k}(C_n) = 2$.
        \item If $n \geq 1$, then $\ch{k}(K_n) = \left \lceil \frac{n}{2} \right \rceil$.
    \end{enumerate}
\end{proposition}

\begin{proof}
We will prove these separately below:
\begin{description}
    \item[Proof of 1:] The winning strategy for the cop is to start in an end-vertex of $P_n$ and then move to the other end-vertex.
    \item[Proof of 2:]  If $n \geq 3$, then $c_0(C_n) = 2$ (see \cite{tosic1986-0-vis}). Thus by Observation \ref{obs:basic-bounds}, $\ch{k}(C_n) \leq 2$. Similarly, $c(C_n) = 2$ \cite[Lemma 1.1]{BoNo11}, thus $\ch{k}(C_n) \geq 2$.
    \item[Proof of 3:] It follows from Observation \ref{obs:0-vis}, the fact that $c_0(K_n) = \lceil \frac{n}{2} \rceil$ (\cite{tosic1986-0-vis}), and since $\diam(K_n) = 1$ that $\ch{k}(K_n) = \lceil \frac{n}{2} \rceil$. \qedhere
\end{description}
\end{proof}

The following is a generalization of~\cite[Theorem 6]{bonato-2019-hyperopic}.

\begin{theorem}
    \label{thm:diameter-bound}
    Let $G$ be a graph and let $k \geq 1$. If $\diam(G) \geq 2k+1$, then $\ch{k}(G) \leq c(G) + 2$.
\end{theorem}

\begin{proof}
    Let $u, v\in V(G)$ be such that $d(u,v) = \diam(G)$. Since $\diam(G) \geq 2k+1$, it holds that for every $x \in V(G)$, $d(x,u) > k$ or $d(x,v) > k$. Placing one cop on $u$ and one on $v$ thus results in robber always being visible during the game. Thus $c(G)$ remaining cops can win the game.
\end{proof}

The equality case for $k=1$ presented in~\cite{bonato-2019-hyperopic} can be generalized to provide sharpness examples of Theorem~\ref{thm:diameter-bound}. 

\begin{example}
    The graph $G_k$ is obtained in the following way. Let $X_1, X_2, X_3$ be sets of size $m \geq 3$, and let $X_1 \cup X_2 \cup X_3$ be a clique. Add paths $P_k^i = v_1^i, \ldots, v_k^i$ for $i \in [3]$ and make $v_1^i$ adjacent to every vertex of $X_i$.

\begin{figure}[!ht]
\centering
\begin{tikzpicture} [scale=.4]
\draw  (-1.25,7.5) ellipse (1.75cm and 1cm);
\draw  (2.5,12.25) ellipse (1.75cm and 1cm);
\draw  (6.25,7.5) ellipse (1.75cm and 1cm);
\node [font=\LARGE] at (-1.25,7.5) {$X_1$};
\node [font=\LARGE] at (6.25,7.5) {$X_2$};
\node [font=\LARGE] at (2.5,12.25) {$X_3$};
\draw  (1,11.75) -- (-1.25,8.5);
\draw  (0.5,7.5) -- (4.5,7.5);
\draw  (3.75,11.5) -- (6,8.5);
\draw  (0.25,8) -- (2.25,11.25);
\draw  (3,11.25) -- (4.75,8);
\draw  (0,6.75) -- (5,6.75);
\draw  (-0.5,8.4) -- (1.5,11.4);
\draw  (3.25,11.4) -- (5.25,8.3);
\draw  (0.5,7.75) -- (4.5,7.75);
\draw  (0.25,7) -- (4.75,7);
\draw  (4,11.75) -- (6.5,8.5);
\draw  (0.75,12) -- (-1.75,8.5);
\draw  (2.5,15.25) circle (0.25cm) node [label=above:{ $v^3_1$}] {};
\draw  (3.75,15.25) circle (0.25cm) node [label=above:{ $v^3_2$}] {};
\draw  (5,15.25) circle (0.25cm) node [label=above:{ $v^3_3$}] {} ;
\draw  (6.25,15.25) circle (0.25cm) node [label=above:{ $v^3_4$}] {} ;
\draw  (2.5,15) -- (2.5,13.25);
\draw  (2.5,15) -- (1.25,13);
\draw  (2.5,15) -- (3.75,13);
\draw  (2.75,15.25) -- (3.5,15.25);
\draw  (4,15.25) -- (4.75,15.25);
\draw  (5.25,15.25) -- (6,15.25);
\node [font=\LARGE] at (7.25,15.25) {$...$};
\draw  (8.25,15.25) circle (0.25cm) node [label=above:{$v^3_k$}] {};
\draw  (10,7.5) circle (0.25cm) node [label=above:{ $v^2_1$}] {} ;
\draw  (11.25,7.5) circle (0.25cm) node [label=above:{ $v^2_2$}] {} ;
\draw  (12.5,7.5) circle (0.25cm) node [label=above:{ $v^2_3$}]  {};
\draw  (13.75,7.5) circle (0.25cm) node [label=above:{ $v^2_4$}] {} ;
\draw  (15.75,7.5) circle (0.25cm) node [label=above:{ $v^2_k$}] {} ;
\draw  (8,7.5) -- (9.75,7.5);
\draw  (10.25,7.5) -- (11,7.5);
\draw  (11.5,7.5) -- (12.25,7.5);
\draw  (12.75,7.5) -- (13.5,7.5);
\node [font=\LARGE] at (14.75,7.5) {...};
\draw  (-5,7.5) circle (0.25cm) node [label=above:{ $v^1_1$}] {} ;
\draw  (-6.25,7.5) circle (0.25cm) node [label=above:{ $v^1_2$}] {} ;
\draw  (-7.5,7.5) circle (0.25cm) node [label=above:{ $v^1_3$}] {} ;
\draw  (-8.75,7.5) circle (0.25cm) node [label=above:{ $v^1_4$}] {} ;
\draw  (-10.75,7.5) circle (0.25cm) node [label=above:{ $v^1_k$}] {} ;
\draw  (-4.75,7.5) -- (-3,7.5);
\draw  (-6,7.5) -- (-5.25,7.5);
\draw  (-7.25,7.5) -- (-6.5,7.5);
\draw  (-8.5,7.5) -- (-7.75,7.5);
\node [font=\LARGE] at (-9.75,7.5) {...};
\draw  (-5,7.75) -- (-2.75,8);
\draw  (-5,7.25) -- (-2.75,7);
\draw  (10,7.25) -- (7.5,6.75);
\draw  (10,7.75) -- (7.5,8.25);
\end{tikzpicture}
\label{fig:triangle}
\caption{ A schematic drawing of the graph $G_k$.}
\end{figure}
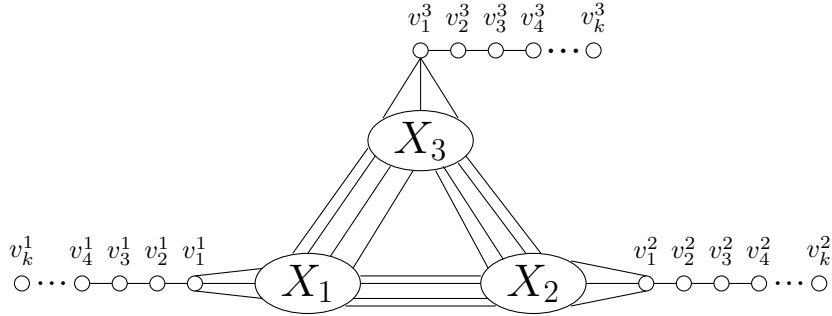

    It is easy to see that $\diam(G_k) = 2k+1$ and that $c(G_k) = 1$. Thus by Theorem \ref{thm:diameter-bound} $\ch{k}(G_k) \leq 3$. In the following we argue that it is actually equal to 3.

    We consider the $k$-hyperopic cops and robber game on $G_k$ with two cops and show that robber has a winning strategy. Robber plays so that after each round, one of the following is achieved.
    \begin{enumerate}
        \item[(i)] If cops are on $P_k^i$ and $P_k^j$, where $i \neq j$, then the robber is on $X_k$, where $k \in [3] \setminus \{i,j\}$. 
        \item[(ii)] If one cop is on $X_1 \cup X_2 \cup X_3$ and the other cop is on a path, say on $P_k^i$, then the robber is on $X_i$.
        \item[(iii)] If both cops are on $X_1 \cup X_2 \cup X_3$, then the robber is also on some vertex of $X_1 \cup X_2 \cup X_3$.
    \end{enumerate}

    In (i), the robber is either invisible, or if robber is visible, the capture is not possible in the next round. In (ii) and (iii), robber is invisible, thus capture is also not possible in the next round (as $m \geq 3$). Robber can start the game such that one of the above states is satisfied. Since all robbers positions described in (i), (ii) and (iii) are pairwise adjacent, robber can easily transition into the correct state after each move of the cops.
\end{example}

A \emph{weak homomorphism} $f \colon G \to H$ is a map from $V(G)$ to $V(H)$ such that if $uv \in E(G)$, then $f(u) f(v) \in E(H)$ or $f(u) = f(v)$. A weak homomorphism $f \colon G \to H$ is a \emph{retraction} if $f(v) = v$ for every $v \in V(G)$. If there is a retraction from $G$ to $H$ then $H$ is a \emph{retract} of $G$. We obtain an analogous relation between graphs and their retracts as in the usual cops and robber game.

\begin{theorem}
    \label{thm:retract}
    If $H$ is a retract of $G$, then $\ch{k}(H) \leq \ch{k}(G)$.
\end{theorem}

\begin{proof}
    Let $f \colon G \to H$ be a retraction and let $\ch{k}(G) = t$. We prove that the \emph{shadow strategy} is a winning strategy for $t$ cops on $H$. Two games are considered in parallel: the game on $G$ and the game on $H$. The robber plays optimally on $H$, and his moves are copied to $G$. Cops play optimally on $G$ according to their strategy, and their moves are copied to $H$ using $f$: if a cop moves from $c$ to $c'$ on $G$, his move on $H$ is from $f(c)$ to $f(c')$ (this is legal as $f$ is a retraction). Note that this means that robber is always positioned on $V(H)$ in both games.

    As $t$ cops have a winning strategy on $G$, one cop, located on $x \in V(G)$, catches the robber in some step of the game on $G$ by moving to the vertex $y \in V(H)$ that the robber occupies. By the above strategy, this cop moves from $f(x)$ to $f(y)$ in this step on $H$. As the robbers' position on $H$ is $y = f(y)$, the cops win on $H$.
\end{proof}

We know from~\cite{bonato-2019-hyperopic} that $\ch{1}(G) = 1$ if and only if $G$ is a tree. The result for $k \geq 2$ is significantly different.

\begin{theorem}
    \label{thm:one}
    If $k \geq 2$ and $G$ is a graph, then $\ch{k}(G) = 1$ if and only if $G$ is a caterpillar.
\end{theorem}

\begin{proof}
    We first show that if $\ch{k}(G) = 1$, then $G$ is a tree. Let $\ch{k}(G) = 1$. If $G$ is not a tree, then by~\cite[Theorem 1]{bonato-2019-hyperopic}, $\ch{1}(G) \geq 2$, thus by Observation~\ref{obs:basic-bounds}, $\ch{k}(G) \geq 2$. 
    \begin{enumerate}
    \item[$\Rightarrow$] 
    Let $\ch{k}(G) = 1$. Suppose that $G$ is a tree, but not a caterpillar. Since removing the leaves of $G$ does not yield a path, we conclude that $G$ contains a subgraph isomorphic to a tree $\widehat{T}$. As $k \geq 2$ and the robber is omniscient, one cop cannot win the game on $\widehat{T}$. Indeed, if $x$ is the vertex of degree 3 in $\widehat{T}$, then robber can always move to a vertex of degree 2 in $\widehat{T}-x$ that neither contains the cop nor is in the component into which the cop will move next. However, it is easy to see that two cops can win on $\widehat{T}$, thus $\ch{k}(\widehat{T}) = 2$. Moreover, $\widehat{T}$ is a retract of $G$, thus by Theorem \ref{thm:retract} $\ch{k}(G) \geq \ch{k}(\widehat{T}) = 2$.  
    Thus  $\ch{k}(G) \geq 2$. This contradicts our assumption that $\ch{k}(G) = 1$, so $G$ must be a caterpillar.
    \item[$\Leftarrow$] 
    Let $G$ be a caterpillar. Note that $c_0(G)=1$ \cite{tosic1986-0-vis} for caterpillars. Thus, by Observation \ref{obs:basic-bounds}, $\ch{k}(G) = 1$ for caterpillars. \qedhere
\end{enumerate}
\end{proof}

\section{General upper bound}
\label{sec:general-upper}

In this section, we present a general upper bound for the $k$-hyperopic cop number, and additionally obtain a new result for the 0-visibility cops and robber game.

\begin{theorem}
    \label{thm:upper-c0}
    If $G$ is a graph, then $$c_0(G) \leq \begin{cases}
        \alpha'(G), & G \text{ has a perfect matching};\\
        \alpha'(G) + 1, & \text{otherwise}.
    \end{cases}$$
    Moreover, if $G$ has $n$ vertices, then $c_0(G) \leq \left \lceil \frac{n}{2} \right \rceil$.
\end{theorem}

\begin{proof}
    Let $G$ be a graph on $n$ vertices. Let $\alpha'(G) = m$ and let $M \subset E(G)$ be a maximum matching in $G$, $|M| = m$. Denote $M = \{x_1 y_1,  \ldots, x_m y_m\}$. If $G$ has a perfect matching, then $m = \frac{n}{2}$, and the winning strategy of $m$ cops is the following. For $i \in [m]$, cop $c_i$ starts on $x_i$ and alternates his position between $x_i$ and $y_i$ in every round. Thus, if the robber is located on $y_i$ (or $x_i$) and not yet caught, the cop $c_i$ will catch the robber in the next round. This gives the upper bound of $m = \frac{n}{2}$. 

    If $G$ does not have a perfect matching, then $I = V(G) \setminus \{x_1, y_1,  \ldots, x_m, y_m\}$ is not empty. Since $M$ is maximal, $G[I]$ contains no edges. The winning strategy of $m+1$ cops is the following. For $i \in [m]$, cop $c_i$ starts on $x_i$ and alternates his position between $x_i$ and $y_i$ in every round. As before it follows that if the robber ever steps on a vertex from $\{x_1, y_1, \ldots, x_m, y_m\}$, they are caught. Thus robber's motion is limited to $G[I]$, which in turn means that unless robber is caught already, he is stationary on a vertex from $I$. The strategy of cop $c_{m+1}$ is simply to visit all vertices from $I$ (which is possible as $G$ is connected). This gives the upper bound of $m+1$, which is at most $\left \lceil \frac{n}{2} \right \rceil$.
\end{proof}

\begin{corollary}
    \label{cor:upper-ch}
    If $G$ is a graph and $k \geq 1$, then $$\ch{k}(G) \leq \begin{cases}
        \alpha'(G), & G \text{ has a perfect matching};\\
        \alpha'(G) + 1, & \text{otherwise}.
    \end{cases}$$
    Moreover, if $G$ has $n$ vertices, then $\ch{k}(G) \leq \left \lceil \frac{n}{2} \right \rceil$.
\end{corollary}

Note that this is a generalization of~\cite[Theorem 3]{bonato-2019-hyperopic} which states that $c_H(G) \leq \lceil \frac{n}{2} \rceil$ for all graphs on $n$ vertices. However, the proof technique is different as domination is used to prove the bound for $c_H$. 

Observe also that it follows from Observation \ref{obs:basic-bounds} and the known equality cases for the 1-hyperopic case from \cite{bonato-2019-hyperopic}, that $K_n$ if $n \geq 1$ and $K_m - e$ if $m \geq 2$ is even, are equality cases also for Corollary \ref{cor:upper-ch}.

\section{Trees}
\label{sec:trees}

It follows from \cite[Theorem 1]{bonato-2019-hyperopic} that $\ch{1}(T) = 1$ for all trees $T$, but we already know from Theorem \ref{thm:one} that there are trees $T$ with $\ch{k}(T) > 1$ if $k \geq 2$. In this section we give several upper bounds for the $k$-hyperopic cop number of trees. If $k=2$, then the problem is relatively simple, and the general upper bound is 2. However, if $k \geq 3$, the problem appears to be much more involved, and our bounds depend mostly on the relation between $k$ and the diameter of the tree.

\begin{theorem}
    \label{thm:2-hyperopic-trees}
    If $T$ is a tree, then $\ch{2}(T) \leq 2$.
\end{theorem}

\begin{proof}
    We present a winning strategy for two cops, $c_1$ and $c_2$ on $T$. They start the game by $c_1$ occupying a leaf and $c_2$ occupying its neighbor. In the $i$th ($i \geq 1$) round of the game their strategy is as follows.
    \begin{enumerate}
        \item[(i)] If robber is visible, then the cops move towards the robber along the shortest path between them.
        \item[(ii)] Otherwise, $c_1$ moves to $c_2$'s position (say $x$), and $c_2$ moves to a neighbor of his current position (different than $c_1$). In the next rounds, $c_1$ is stationary, while $c_2$ returns back to $x$ and then iteratively visits every neighbor of $x$. If the robber becomes visible during this procedure, then $c_2$ moves to a neighbor of $c_1$ that is in the same component of $T-c_1$ as the robber is.
    \end{enumerate}

    First observe that if the cops use the above strategy, then while the game is not over, at the end of each sequence of moves, robber is never positioned in the component of $T-c_2$ containing $c_1$, and the size of the component containing $c_1$ increases by at least one. This is trivially true at the start of the game. Suppose it is true until round $i-1$. 
    
    If robber is visible, then by induction hypothesis, it is contained in a component of $T-c_2$ that does not contain $c_1$, but contains some neighbor of $c_2$, say $y$. Thus after both cops move closer to the robber along the shortest path, $c_1$ moves to the previous position of $c_2$, and $c_2$ moves to $y$. Thus after this move, robber is again not contained in the component of $T-c_2$ containing $c_1$, while the component containing $c_1$ now contains at least one new vertex.
    
    If robber is not visible, then since by the induction hypothesis the robber is not in the connected component of $T-c_2$ containing $c_1$, the robber must be on a vertex from $N(c_2) - \{c_1\}$. During (ii), $c_2$ either catches the robber, or the robber becomes visible, and in this case, at the end of the sequence, robber is again not in the component of $T-c_2$ containing $c_1$, but component containing $c_1$ now contains at least one new vertex.

    Since $T$ is finite, it follows that the above strategy is indeed a winning strategy.
\end{proof}

The bound in Theorem \ref{thm:2-hyperopic-trees} cannot be improved as for example the 2-hyperopic cop number of a subdivision of $K_{1,3}$ is 2 (see the argument in the proof of Theorem \ref{thm:one}).

\begin{theorem}
    If $T$ is a tree with a pendant $P_{k-1}$ then $\ch{k}(G) \leq 2$. 
\end{theorem}

\begin{proof}
    Let $T$ be a tree with a pendant $P_{k-1}$. We provide a winning strategy for the cops. Let $v_1...v_{k-1}$ be the vertices of our pendant path. Place one cop on $v_1$. This cop remains on $v_1$ for the remainder of the game. If the robber is not visible then we know that the robber is on the pendant path and the remaining cop can walk along the path to catch the robber. If the robber is visible then the game becomes the standard cops and robbers game where $c(T)=1$. Thus the remaining cop catches the robber. Hence this is a winning strategy with two cops. 
\end{proof}

We recall the following sequence of families of trees from \cite{dereniowski+2015-0-visibility-2}. The family $\T_1 = \{K_1 \}$. For $m \geq 1$, the family $\T_{m+1}$ consists of all trees $T$ derived from $T_1, T_2, T_2 \in \T_m$ (not necessarily distinct) by adding a disjoint copy of $K_{1,3}$ with vertices $x, y_1, y_2, y_3$ (where $x$ is the center) and making $y_i$ adjacent to one vertex in $T_i$, $i \in [3]$. See for example Figure \ref{fig:family-T}.

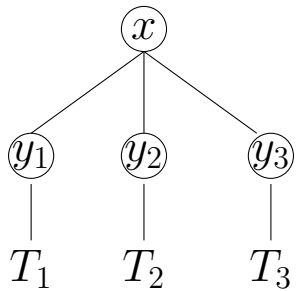
\begin{figure}
\centering
\begin{tikzpicture} [scale=.75]
\tikzstyle{every node}=[font=\LARGE]
\node [font=\LARGE] at (11,16.75) {$T_1$};
\node [font=\LARGE] at (15.25,16.75) {$T_3$};
\node [font=\LARGE] at (13,16.75) {$T_2$};
\node [font=\LARGE] at (11,18.75) {$y_1$};
\node [font=\LARGE] at (13,18.75) {$y_2$};
\node [font=\LARGE] at (15.25,18.75) {$y_3$};
\node [font=\LARGE] at (13,21) {$x$};
\draw  (13,21) circle (.4cm);
\draw  (11,18.75) circle (.4cm);
\draw  (13,18.75) circle (.4cm);
\draw  (15.25,18.75) circle (.4cm);
\draw   (13,20.6) -- (11,19.15);
\draw   (13,20.6) -- (13,19.15);
\draw   (13,20.6) -- (15,19.15);
\draw   (11,18.25) -- (11,17.25);
\draw   (13,18.25) -- (13,17.25);
\draw   (15.25,18.25) -- (15.25,17.25);
\end{tikzpicture}
\caption{A construction of a tree in $\T_{m+1}$ from $T_1, T_2, T_3 \in \T_m$.}
\label{fig:family-T}
\end{figure}

\begin{lemma}
    \label{lem:t_k-diameter}
    If $T \in \T_m$, then $\diam(T) \geq 4m$.
\end{lemma}

\begin{proof}
    Let $d_m = \min \{ \diam(T) \mid T \in \T_m \}$. Clearly, $d_1 = 0$. Since $K_{1,3}$ with all three edges subdivided once (so $\widehat{T}$ from the proof of Theorem \ref{thm:one}) is the only member of $\T_2$, $d_2 = 4$. By the definition of $\T_m$, $m \geq 2$, we have $$d_m = 2 \min \{ \rad(T) \mid T \in \T_{m-1} \} + 4 \geq 2 \cdot \frac{d_{m-1}}{2} + 4 = d_{m-1} + 4,$$
    where the inequality follows from the fact that $\diam(G) \leq 2 \rad(G)$ for every graph $G$.

    We prove that $d_m \geq 4m$ by induction on $m$. The statement holds for $m = 1$ and $m = 2$. Let $m \geq 3$. Then we have $$d_m \geq d_{m-1} + 4 \overset{\text{I.H.}}\geq 4 (m-1) + 4 = 4m.$$
    Since for every $T \in \T_m$, $\diam(T) \geq d_m$, the claim follows.
\end{proof}

\begin{theorem}[{\cite[Theorem 3.8]{dereniowski+2015-0-visibility-2}}]
    \label{thm:trees-0-visibility}
    If $T$ is a tree, then $c_0(T) \geq m$ if and only if there is $T' \in \T_m$ such that $T'$ is a minor of $T$.
\end{theorem}

This results means that there are trees with arbitrarily large $k$-hyperopic cop number (at least for specific values of $k$). In particular, let $T$ be a tree from $\T_m$. Then $c_0(T) = m$ and thus $\ch{k}(T) = m$ for all $k \leq \diam(T)$ by Observation \ref{obs:0-vis}. 

Theorem \ref{thm:trees-0-visibility} helps us prove the following general upper bound for the 0-visibility cop number (and in turn for the $k$-hyperopic cop number) of trees ($k \geq 3$).

\begin{theorem}
    \label{thm:trees-diam/4-0-vis}
    If $T$ is a tree and $\diam(T) \geq 4$, then $$c_0(T) \leq \left \lfloor \frac{\diam(T)}{4} \right \rfloor.$$
\end{theorem}

\begin{proof}
    Suppose that $c_0(T) \geq \left \lfloor \frac{\diam(T)}{4} \right \rfloor + 1$. Then by Theorem \ref{thm:trees-0-visibility} $T$ contains some tree $T' \in \T_{\left \lfloor \frac{\diam(T)}{4} \right \rfloor + 1}$ as a minor. Since $T'$ is a minor of $T$, $\diam(T) \geq \diam(T')$. By Lemma \ref{lem:t_k-diameter}, $\diam(T') \geq 4 (\left \lfloor \frac{\diam(T)}{4} \right \rfloor) \geq \diam(T) + 1$, which is a contradiction.
\end{proof}

\begin{corollary}
    \label{cor:trees-diam/4-hyperopic}
    If $T$ is a tree, $\diam(T) \geq 4$ and $k \geq 1$, then $$\ch{k}(T) \leq \left \lfloor \frac{\diam(T)}{4} \right \rfloor.$$
\end{corollary}

For completeness we also note that if $T$ is a tree and $\diam(T) \leq 3$, then $\ch{k}(T) = c_0(T) = 1$ for all $k \geq 1$.

Since the cop number of trees is 1 \cite[Lemma 1.2]{BoNo11}, Theorem \ref{thm:diameter-bound} implies that trees $T$ with $\diam(T) \geq 2k-1$ have $\ch{k}(T) \leq 3$. However, we can slightly improve this upper bound.

\begin{lemma}
     If $T$ is a tree and $\diam(T) \geq 2k-1$, then $\ch{k}(T) \leq 2$.
\end{lemma}

\begin{proof}
   Let $T$ be a tree with $\diam(T) \geq 2k-1$ and $v, x_1, \ldots, x_m, u$, where $m \geq 2k-2$, be the vertices along the path that achieves this diameter. Place $c_1$ on vertex $v$ and $c_2$ on vertex $x_k$. We can now split the graph into two components along the edge between $x_{k-1}$ and $x_{k}$. We will call the connected component of $T-x_{k-1} x_k$ containing $x_{k}$, component $A$, and the other component $B$. If the robber is in $A$, which is fully visible as for every $x \in V(A)$, $d(v, x) \geq k$ except for $x_k$ which is covered because of $c_2$, then only one cop is needed to catch the robber and $c_2$ moves along the shortest path towards the robber who is caught.
   
    If the robber is in B, this means that the robber is invisible, unless the robber is along some path attached to $x_i$ for $i\in[k-1]$ greater than distance $k$ away from $v$ or $x_k$. In this case we move $c_1$ to vertex $u$, thus making component $B$ visible (as every $y \in V(B)$ is at distance at least $k$ from $u$) and component A invisible (with the exception analogous of the paths stated before). However, since $c_2$ has stayed on vertex $x_{k}$ the whole time, the robber cannot leave component B, and thus now can be caught since only one cop is needed for the capture. 
\end{proof}

However, it turns out that the $k$-hyperopic cop number is bounded by 3 for trees with slightly smaller diameter as well. The following bounds give a weaker bound the closer the diameter is to $k$, which is in line with the fact that for trees $T$ with $\diam(T) = k$, the $k$-hyperopic cop number equals the 0-visibility number (see Observation \ref{obs:0-vis}).

\begin{theorem}
    \label{thm:upper-bound-trees-different-diameters-close-to-2k}
    If $T$ is a tree with $2k-3 \leq \diam(T) \leq 2k-2$, then $$\ch{k}(T) \leq 3.$$
\end{theorem}

\begin{proof}
    Let $d = \diam(T)$ and let $x, y \in V(T)$ be such that $d(x,y) = d$.  
    
    We place $c_1$ on $x$ and $c_2$ on $y$ and they stay stationary throughout the game. This means that some parts of the tree are always visible. More precisely, the invisible vertices are $I = \{v \in V(T) \mid d(v,x) \leq k, d(v,y) \leq k \}$. Equivalently, $I = \{ v \in V(T) \mid d(v, C(T)) \leq k - \lfloor \frac{d}{2} \rfloor \}$. If $|C(T)| = 1$, then $d$ is even, and $\diam(T[I]) \leq 2 (k - \frac{d}{2}) = 2 k - d$. Otherwise, so if $|C(T)| = 2$, then $d$ is odd and $\diam(T[I]) \leq 2 (k - \frac{d-1}{2}) + 1 = 2 k - d$. 

    Since $2k-3 \leq d \leq 2k-2$, we have $\diam(T[I]) \leq 3$. Thus $T[I]$ is $K_1$, $K_2$, a star, or a star with one edge subdivided. Since the 0-visibility cop number of all these trees is 1, $c_3$ can catch the robber in the invisible part. If the robber is visible, then $c_3$ simply moves towards the robber and wins the game.
\end{proof}

\begin{theorem}
    \label{thm:upper-bound-trees-different-diameters-2}
    If $T$ is a tree with $k+1 \leq \diam(T) \leq 2k-4$, then $$\ch{k}(T) \leq 2 + \left \lfloor \frac{k}{2} - \frac{\diam(T)}{4} \right \rfloor \leq 2 + \left \lfloor \frac{k}{2} \right \rfloor - \left \lfloor  \frac{\diam(T)}{4} \right \rfloor.$$
\end{theorem}

\begin{proof}
    Let $d = \diam(T)$ and let $x, y \in V(T)$ be such that $d(x,y) = d$. As in the proof of Theorem \ref{thm:upper-bound-trees-different-diameters-close-to-2k}, we place $c_1$ on $x$, $c_2$ on $y$, define the set of invisible vertices $I$, and observe that $\diam(T[I]) \leq 2 k - d$. If $\diam(T[I]) < 2k-d$, we add some vertices to $I$ to ensure $\diam(T[I]) = 2k-d$ (this is possible since $d \geq k+1$). Observe that $2k-d \geq 4$ as $d \leq 2k-4$. 

    The strategy of $2 + \lfloor \frac{2k-d}{4} \rfloor$ cops is the following. Two cops ($c_1$, $c_2$) are positioned on $x$, $y$ throughout the game, respectively. The remaining $\lfloor \frac{2k-d}{2} \rfloor \geq 2$ cops either move towards the robber (if the robber is visible), or play their optimal strategy for the $k$-hyperopic game on $T[I]$ (if the robber is invisible) which exists by Corollary \ref{cor:trees-diam/4-hyperopic}.
\end{proof}

Combining the above results we see that for a given fixed $k$, if diameter of a tree is at most $k$, then we are in the 0-visibility case. This results in the $k$-hyperopic cop number being arbitrarily large. On the other hand, if diameter is at least $2k-3$ then the upper bound is 2 or 3. It is unclear how the $k$-hyperopic cop number behaves on trees with diameter between $k+1$ and $2k-4$. The upper bound from Theorem \ref{thm:upper-bound-trees-different-diameters-2} bounds the $k$-hyperopic cop number in terms of $k$ but, as diameter gets closer to $k+1$, the bound becomes larger and larger. The below example shows however that there are trees with diameter $k+1$ that have small $k$-hyperopic cop number as well.

\begin{example}
We provide an example of a tree $T$ with $\diam(T) = k+1$ such that $\ch{k}(T)=3$. Let $k = 9$. The tree $T$ is obtained from $\widehat{T}$ by attaching a $K_{1,3}$ where each edge is subdivided twice to each leaf of $\widehat{T}$. See Figure \ref{fig:example-tree}. Let $\mathcal{L}$ denote the set of all leaves of the obtained tree $T$. Clearly, $\diam(T) = 10 = 9+1$. More precisely, the only vertices at distance 10 are some pairs of leaves. Thus, the 9-hyperopic cops and robber game on the $T - \mathcal{L}$ is equivalent to the 0-visibility game on $T - \mathcal{L}$. As $T - \mathcal{L} \in \T_3$, we have $c_0(T - \mathcal{L}) = 3$. As additional visibility of leaves $\mathcal{L}$ only helps the cops to catch the robber, $c_{H,9}(T)\leq3$. But as $V(T) - \mathcal{L}$ always remains invisible and the robber can escape two cops there, the same strategy ensures that two cops cannot catch the robber on $T$. Hence, $\ch{9}(T) = 3$. On the other hand, Theorem \ref{thm:upper-bound-trees-different-diameters-2} gives only the bound $\ch{9}(T) \leq 4$. 

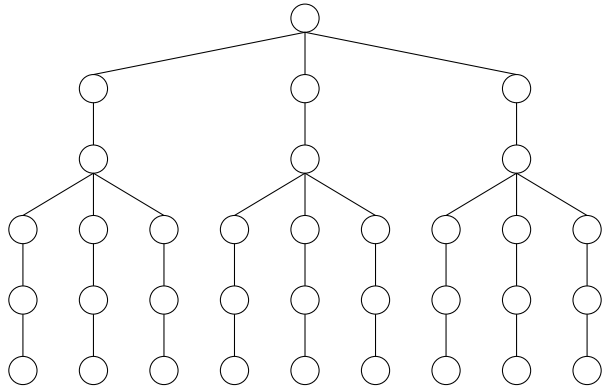
\begin{figure}[!ht]
\centering
\begin{tikzpicture} [scale=.75]
\draw  (7.5,11.5) circle (0.25cm);
\draw  (7.5,10.25) circle (0.25cm);
\draw  (7.5,9) circle (0.25cm);
\draw  (7.5,7.75) circle (0.25cm);
\draw  (7.5,6.5) circle (0.25cm);
\draw  (7.5,5.25) circle (0.25cm);
\draw (7.5,11.25) -- (7.5,10.5);
\draw  (7.5,10) -- (7.5,9.25);
\draw (7.5,8.75) -- (7.5,8);
\draw  (7.5,7.5) -- (7.5,6.75);
\draw  (6.25,7.75) circle (0.25cm);
\draw  (6.25,6.5) circle (0.25cm);
\draw  (6.25,5.25) circle (0.25cm);
\draw  (8.75,7.75) circle (0.25cm);
\draw  (8.75,6.5) circle (0.25cm);
\draw  (8.75,5.25) circle (0.25cm);
\draw  (2.5,7.75) circle (0.25cm);
\draw  (3.75,7.75) circle (0.25cm);
\draw  (5,7.75) circle (0.25cm);
\draw  (2.5,6.5) circle (0.25cm);
\draw  (3.75,6.5) circle (0.25cm);
\draw  (5,6.5) circle (0.25cm);
\draw  (2.5,5.25) circle (0.25cm);
\draw  (3.75,5.25) circle (0.25cm);
\draw  (5,5.25) circle (0.25cm);
\draw  (10,7.75) circle (0.25cm);
\draw  (11.25,7.75) circle (0.25cm);
\draw  (12.5,7.75) circle (0.25cm);
\draw  (10,6.5) circle (0.25cm);
\draw  (11.25,6.5) circle (0.25cm);
\draw  (12.5,6.5) circle (0.25cm);
\draw  (10,5.25) circle (0.25cm);
\draw  (11.25,5.25) circle (0.25cm);
\draw  (12.5,5.25) circle (0.25cm);
\draw  (3.75,10.25) circle (0.25cm);
\draw  (3.75,9) circle (0.25cm);
\draw  (11.25,10.25) circle (0.25cm);
\draw  (11.25,9) circle (0.25cm);
\draw (7.5,6.25) -- (7.5,5.5);
\draw  (8.75,6.25) -- (8.75,5.5);
\draw  (8.75,7.5) -- (8.75,6.75);
\draw  (10,7.5) -- (10,6.75);
\draw  (10,6.25) -- (10,5.5);
\draw  (11.25,6.25) -- (11.25,5.5);
\draw  (12.5,6.25) -- (12.5,5.5);
\draw  (12.5,7.5) -- (12.5,6.75);
\draw  (11.25,7.5) -- (11.25,6.75);
\draw  (11.25,8.75) -- (11.25,8);
\draw  (11.25,10) -- (11.25,9.25);
\draw  (6.25,7.5) -- (6.25,6.75);
\draw  (6.25,6.25) -- (6.25,5.5);
\draw  (5,6.25) -- (5,5.5);
\draw  (5,7.5) -- (5,6.75);
\draw  (3.75,7.5) -- (3.75,6.75);
\draw  (3.75,6.25) -- (3.75,5.5);
\draw  (2.5,6.25) -- (2.5,5.5);
\draw  (2.5,7.5) -- (2.5,6.75);
\draw  (2.5,8) -- (3.75,8.75);
\draw  (3.75,8.75) -- (3.75,8);
\draw  (3.75,8.75) -- (5,8);
\draw  (3.75,10) -- (3.75,9.25);
\draw  (7.5,8.75) -- (6.25,8);
\draw  (7.5,8.75) -- (8.75,8);
\draw  (11.25,8.75) -- (10,8);
\draw  (11.25,8.75) -- (12.5,8);
\draw  (3.75,10.5) -- (7.5,11.25);
\draw  (7.5,11.25) -- (11.25,10.5);
\end{tikzpicture}
\caption{A tree with diameter 10 and 9-hyperopic cop number equal to 3.} 
\label{fig:example-tree}
    \end{figure}
\end{example}

\section{Outerplanar}
\label{sec:outerplanar}

In \cite{bonato-2019-hyperopic} they prove that the hyperopic cop number of outerplanar graphs is at most 2, and the hyperopic cop number of planar graphs is at most 3, thus matching the bounds holding for the usual cop number. In this section we prove that the 2-hyperopic cop number of outerplanar graphs is at most 2, but are unable to provide a satisfactory upper bound for the $k$-hyperopic cop number of outerplanar graphs if $k \geq 3$. The idea of the proof for the 2-hyperopic case was inspired by \cite[Theorem 4.1.1]{clarke-2002-phd}, but additional methods were needed to deal with the limited visibility of cops, and the proof is significantly different than in the 1-hyperopic case.

\begin{theorem}
    \label{thm:outerplanar-2}
    If $G$ is an outerplanar graph,  then $\ch{2}(G) \leq 2.$
\end{theorem}

\begin{proof}
    First suppose that $G$ is 2-connected. Then as in \cite{clarke-2002-phd} we can assume that $V(G) = \{v_0, \ldots, v_{n-1}\}$ and $v_i v_{i+1} \in E(G)$ for all $i \in [n]$ (subscripts taken modulo $n$). We fix an outerplanar embedding of $G$ such that $v_0 v_1 \ldots v_{n-1}$ form an outer cycle in the clockwise direction.

    If there are no chords in the embedding, then $G = C_n$ and by [add reference] $\ch{2}(G) = 2$. Let $x_0, x_1, \ldots, x_k$ be the vertices of degree at least 3 in clockwise direction around the cycle. Note that vertices on the cycle between $x_i$ and $x_{i+1}$ are of degree 2 so there is a unique path on the cycle between $x_i$ and $x_{i+1}$. 

    For $x_i, x_j$, $i<j$, let $C_{i,j}$ denote all vertices on the outer cycle between $x_i$ and $x_j$ in the clockwise direction (excluding $x_i$ and $x_j$). If two cops are on $x_i$ and $x_j$, robber is not on $C_{i,j}$ and there is no chord between $C_{i,j}$ and $C_{j,i}$, then the two cops are guarding $C_{i,j} \cup \{x_i, x_j\}$. We call such $C_{i,j}$ the cops' territory, and say that $|C_{i,j}|$ is its size (even though two additional vertices, $x_i$ and $x_j$, are guarded by it).

    We will prove that the cops can establish a cops' territory and then keep making it larger, thus eventually winning the game (since $G$ is finite).

    We start by proving that cops can establish a cops' territory of size at least two starting from some vertex $v \in V(G)$. Without loss of generality let $v \in C_{0,1} \cup \{x_0\}$, and if $\deg(v) \geq 3$, then $C_{0,1}$ is the larger between the two paths of degree 2 vertices incident with $v$. We distinguish between the following cases.
    
    \begin{description}
        \item[Case A:] $|C_{0,1}| \geq 2$.\\
        In this case, $c_1$ moves from $v$ to $x_0$ (possibly staying on the same vertex if $v = x_0$) while $c_2$ moves from $v$ to $x_1$. After that, $C_{0,1}$ is the cops' territory since by definition there is no chord between $C_{0,1}$ and $C_{1,0}$. In this case, the obtained cops territory contains at least two vertices, and $v$ is within it.

        \item[Case B:] $|C_{0,1}| = 1$.\\
        Let $w$ be the neighbor of $x_0$ in the counterclockwise direction on the outer cycle (possibly $w = x_k$), and let $x_m \in C_{1,0}$ be a neighbor of $x_1$ that is closest to $x_0$ in the clockwise direction. Since $x_k$ and $x_0$ are both of degree at least 3, $|C_{m,0}| \geq 2$. Cops first move $c_1$ to $x_0$ and $c_2$ to $x_1$. 
        
        If $w x_m \notin E(G)$, then at the same time, $c_1$ moves to $w$ and back while $c_2$ moves to $x_m$ and back. Due to $c_2$, during these rounds, robber cannot move to $x_1$ or $x_m$ and also cannot switch between $C_{1,m}$ and $C_{m,1}$. But due to $c_1$, $C_{1,m}$ is visible. If robber is on $C_{1,m}$, then $c_1$ moves from $x_0$ to $x_m$ and $C_{m,1}$ is cops' territory, containing $v$, and being of size at least 2. Otherwise, robber has to be on $C_{m,0}$, so $c_2$ moves to $x_m$ which again establishes cops' territory containing $v$ and being of size at least two.

        If $w x_m \in E(G)$, then this implies that $w = x_k$ and $x_0 x_m \in E(G)$. Let $C_{0,1} = \{y\}$ and note that $v \in \{x_0, y\}$. Cops now do the following: simultaneously, $c_1$ alternates between $x_0$ and $y$ while $c_2$ alternates between $x_m$ and $x_k$. Thus robber cannot be located on $\{x_0, y, x_k, x_m\}$ and can also not move between $C_{k,m}-\{x_0,y\}$ and $C_{m,k}$. Thus while $c_2$ keeps alternating between $x_k$ and $x_m$ (to prevent robber from moving to a different part of the graph), $c_1$ settles on $y$. Due to $c_1$ all vertices of $C_{m,k}$ are visible. If robber is there, $c_1$ moves to $x_k$ and then $c_2$ settles on $x_m$, thus cops establish cops' territory $C_{k,m}$ of size at least 2, containing $v$. If robber is not visible, robber can only be located on $C_{0,m} - \{y\}$, thus $c_1$ moves to $x_1$, and then $c_2$ settles on $x_m$. Again, cops establish cops' territory $C_{m,1}$ which is of size at least 2 and contains $v$.

        \item[Case C:] $|C_{0,1}| = 0$.\\
        By assumption this implies that $v = x_0$ and $|C_{k,0}| = 0$. Let $x_m \in C_{1,0}$ be a neighbor of $x_1$ that is closest to $x_0$ in the clockwise direction. Since $\deg(x_0), \deg(x_k) \geq 3$, we have $|C_{m,k}| \geq 1$. Cops both start on $x_0=v$, $c_1$ stays there, and $c_2$ moves to $x_1$ in the next round.

        \begin{description}
            \item[Case $C_1$:] $x_k x_m \notin E(G)$\\
            In this case, at the same time, $c_1$ moves from $x_0$ to $x_k$ and back again, while $c_2$ moves from $x_1$ to $x_m$ and back. Note that $x_k x_1 \notin E(G)$ by definition of $x_m$, and $x_k x_m \notin E(G)$ by assumption. This together with the fact that $|C_{m,k}| \geq 1$ implies that while $c_1$ is on $x_k$, the whole $C_{1,k}$ is visible, and due to the movement of $c_2$, the robber cannot be on $x_1, x_m$, and also cannot move between $C_{1,m}$ and $C_{m,1}$. If robber is visible in $C_{1,m}$, then $c_1$ moves to $x_1$ while $c_2$ moves to $x_m$ and $C_{m,1}$ is cops' territory which contains at least two vertices, including $v$. If the robber is invisible, then $C_{0,m}$ is cops' territory, again containing at least two vertices, including $v$.

            \item[Case $C_2$:]$x_k x_m \in E(G)$\\
            In this case, as $\deg(x_0) \geq 3$, we must have $x_0 x_m \in E(G)$. If for every $y \in C_{1,m}$, $y x_m \in E(G)$, then as a consequence, vertices in $C_{1,m}$ are adjacent only to their immediate neighbors on the outer cycle, and $x_m$. The strategy of cops is for $c_1$ to alternate between $x_0$ and $x_m$ (thus preventing the robber to move between $C_{0,m}$ and $C_{m,0}$), while $c_2$ moves to $x_m$ and then counter clockwise from $x_m$ to $x_1$ along the outer cycle. When $c_2$ reaches $x_1$, if the robber was not caught, we know that robber is now on $C_{m,0}$, thus cops have established cops' territory $C_{0,m}$ which has the desired properties. Otherwise, let $u \in C_{1,m}$ be such that $u x_m \notin E(G)$. Now $c_1$ moves from $x_0$ to $x_m$ and back, while $c_2$ moves from $x_1$ to $u$ and back. Due to $c_2$, $C_{m,0}$ is visible, while $c_1$ is preventing the robber from moving between $C_{0,m}$ and $C_{m,0}$. If the robber is visible, then $c_1$ stays on $x_0$ and $c_2$ moves to $x_m$, establishing the desired cops' territory $C_{0,m}$. If the robber is invisible, $c_1$ moves to $x_m$ while $c_2$ stays on $x_1$, again establishing the desired cops' territory $C_{m,1}$.
        \end{description}
    \end{description}

    Now assume that $C_{i,j}$, $i<j$, is cops' territory, $c_1$ is on $x_i$ and $c_2$ is on $x_j$. By above, we can assume that $|C_{i,j}| \geq 2$. We distinguish between the following cases.

    \begin{description}
        \item[Case 1:] There is no chord between $\{x_i, x_j\}$ and $C_{j,i}$.\\
        Cop $c_2$ moves from $x_j$ to $x_{j+1}$; by the assumption and the condition in this case, robber can never enter the part of the outer cycle between $c_1$ and $c_2$. By the end of this sequence of cops' moves, they have enlarged the cops' territory.

        \item[Case 2:] Otherwise; without loss of generality there is a chord between $x_j$ and $C_{j,i}$, and let $x_m \in C_{j,i}$ be the neighbor of $x_j$ closest to $x_i$.\\
        Let $t = |C_{m,i}|$ be the number of vertices between $x_i$ and $x_m$ on the outer cycle in the counterclockwise direction. If $t=0$, then $c_1$ moves to $x_m$ which enlarges the cops' territory. If $t=1$, then by the assumptions, the vertex between $x_i$ and $x_m$ is of degree 2. So if cop $c_1$ moves from $x_i$ to $x_m$ (in two consecutive rounds), then robber cannot enter the part of the outer cycle between $c_1$ and $c_2$, and cops' have again enlarged the cops' territory.

        If $t \geq 2$, we distinguish between two cases. 
        \begin{description}
            \item[Case 2a:] There is a vertex $v \in C_{i,j} \cap N(x_j)$ such that $d(v, x_i) \geq 2$.\\
            Cop $c_2$ moves to $v$ (and then immediately back to $x_j$ in the next round, thus keeping $C_{i,j}$ as cops' territory). At the same time, $c_1$ does the following. If $x_i x_m \in E(G)$, then $x_1$ moves to $x_m$ and back. Otherwise, $x_i$ moves to a neighbor along the outercycle in $C_{i,j}$. Due to $c_1$, $x_m$ is visible during that round, and due to $c_2$ all vertices of $C_{m,i}$ are visible. 

            If robber is in $C_{m,i}$, then $c_2$ moves to $x_m$ which makes cops' territory larger.
            If robber is not in $C_{m,i}$, then in the next rounds, $c_2$ is alternating between $x_j$ and $x_m$ (to prevent robber from entering $C_{i,j}$ and $C_{m,i}$), while $c_1$ moves to $x_m$. Once $c_1$ is on $x_m$, $c_2$ stops alternating and settles on $x_j$, which enlarges cops' territory.

            \item[Case 2b:] Otherwise.\\
            If by exchanging the role of $x_i$ and $x_j$ condition of Case 2a is satisfied, then we use the strategy from Case 2a. 
            Otherwise, $|C_{i,j}| \leq 1$, which is not possible.
        \end{description}
    \end{description}

    In the rest of the proof we consider $G$ which is not 2-connected. Let $\mathcal{B} = \{G_1, \ldots, G_k\}$ be the set of maximal induced 2-connected subgraphs of $G$. Note that each $G_i$ contains at least one cut vertex of $G$, and that each $G_i$ has at least two vertices. 

    Let $\sigma(G_i)$ denote the shadow of $G$ onto $G_i$ obtained by mapping each vertex of $G_i$ to itself and for every cut vertex $x \in V(G_i)$ of $G$ mapping every vertex that is disconnected from $G_i$ in $G-x$ to $x$. Clearly, $\sigma(G_i)$ is a 2-connected outerplanar graph. Notice that if the robber plays on $G$ and cops play on some $G_i$ (using $\sigma(G_i)$ to map robbers position), then if the cops win on $G_i$, they either catch the robber or they catch the shadow of the robber on a cut vertex $x \in V(G_i)$, which means that robber is positioned on the part of $G$ that becomes disconnected from $G_i$ in $G-x$.

    If there is an $i \in [k]$ such that for every cut vertex $x \in V(G_i)$ of $G$ there exists $v \in V(G_i)$ such that $d(x, v) \geq 2$, then cops start the game by playing on $G_i$, using above strategy for 2-connected outerplanar graphs and the shadow of the robber $\sigma(G_i)$. In this way they either catch the robber or robber's shadow. In the latter case, robber's shadow is on a cut vertex $x \in V(G_i)$ of $G$. One of the cops moves to the vertex $v \in V(G_i)$ that is at distance at least 2 from $x$, which ensures that in the part of the graph that contains the robber (the part disconnected from $G_i$ in $G-x$) every vertex is visible, while the other cops stays on $x$. Thus cops now know in which component of $G-x$ robber is, and they can continue playing there by using the combination of the shadow strategy and the strategy on a 2-connected outerplanar graph. Notice that the strategy of establishing the initial cops' territory in the 2-connected case implies that $x$ will be within this initial cops' strategy.

    Otherwise, so if for every $i \in [k]$, $G_i$ contains a cut vertex of $G$ which is adjacent to all other vertices of $V(G_i)$, then every block $G_i$ is isomorphic to a path with an additional vertex that is adjacent to all vertices of the path. Observe that 2-hyperopic number of such graphs is 2 (one cop is on the universal vertex, the other moves along the whole path). 
    
    Let $G_1$ be a block of $G$ with only one cut vertex (it exists due to the definition of blocks) and let $G_2$ be the incident block. Since $G_1$ has a cut vertex that is adjacent to all other vertices of $V(G_1)$, say $x$, $c_1$ can be on $x$ while $c_2$ visits all other vertices in $G_1$. Thus the two cops catch the robber or robber's shadow on $G_1$. If the game is not over yet, robber's shadow is on $x$. Now $c_1$ alternates between $x$ and $y$, where $y$ is the universal vertex of $G_2$, while $c_2$ visits all other vertices of $G_2$ along the outer cycle, until robber or robber's shadow is caught, say on vertex $z$. Now cops' territory contains $G_1 \cup G_2$ and that $z \neq x$. Thus there is a vertex $v \in V(G_1) \cup V(G_2)$ such that $d(z,v) \geq 2$. Thus one cop can move to $v$ while the other two wait on $z$ until the vertices that are disconnected from $G_2 \cup G_1$ in $G-z$ are all visible, and then start the strategy on an outerplanar graph that ensures a win of two cops as above.
\end{proof}

If $k \geq 3$, we believe that the $k$-hyperopic cop number can still be bounded from above independently of the number of vertices of the graph, but are unable to prove it. Instead, we provide the following upper bound, which is better that the bound in Corollary \ref{cor:upper-ch} if the matching number of the graph is not too small.

\begin{proposition}
    \label{prop:outerplanar-2con-k}
    If $G$ is a 2-connected outerplanar graph on $n \geq 5$ vertices, then $\ch{k}(G) \leq \sqrt{2n}.$
\end{proposition}

\begin{proof}
    First observe that there are only three nonisomorphic 2-connected outerplanar graphs on 5 vertices and that two ($2 \leq \sqrt{10}$) cops can win even the 0-visibility game on all of them. 
    
    Using the same notation as in the proof of Theorem \ref{thm:outerplanar-2}, relabel the vertices such that $C_{0,1}$ is the largest among all $C_{i, i+1}$. Cops start the game by $c_1$ staying on $x_0$ while $c_2$ moves from $x_0$ to $x_1$ along the outer cycle. 

    Now assume that $C_{i,j}$, $i<j$ is cops' territory, $c_1$ is on $x_i$ and $c_2$ is on $x_j$. By above and since every 2-connected outerplanar graph contains a vertex of degree 2, $|C_{i,j}| \geq 1$. If there is no chord between $x_i$ or $x_j$ and $C_{j,i}$, then as in Case 1 in the proof of Theorem \ref{thm:outerplanar-2}, cops can enlarge the cops' territory. Otherwise, let $x_m \in C_{j,i}$ be the neighbor of $x_j$ closest to $x_i$ along the outer cycle, and let $x_n \in C_{j,i}$ be the neighbor of $x_i$ closest to $x_j$ along the outer cycle. Without loss of generality, $|C_{n,i}| \leq |C_{j,m}|$. Thus, $|C_{n,i}| \leq \frac{n-3}{2} - 1 = \frac{n-5}{2}$. Now, $c_1$ alternates between $x_i$ and $x_n$, $c_2$ stays on $x_j$, and $\sqrt{2\frac{n-5}{2}}$ other cops play the $k$-hyperopic cops and robber game on $G[C_{n,i}]$. If robber is not caught, robber can only be on $C_{j,n}$, thus by moving $c_1$ to $x_n$, cops' enlarge their territory. As $2 + \sqrt{2\frac{n-5}{2}} \leq \sqrt{2n}$ for $n \geq 5$, we are done.
\end{proof}

\section{Future Work}
\label{sec:future}

To conclude the paper we gather several possible directions of further research of the $k$-hyperopic cops and robber game in this section.

As already mentioned in Section \ref{sec:trees}, the $k$-hyperopic cop number of trees is nontrivial for $k \geq 3$. It would be interesting to understand the relation between the $k$-hyperopic cop number and the given $k$ and the diameter of the tree.

As seen in Section \ref{sec:general-upper}, several of the equality cases from \cite{bonato-2019-hyperopic} serve as equality cases for the upper bound for $c_{H,k}$. It would be interesting to know if there are more equality cases. More precisely, we ask the following.

\begin{question}
    Can the graphs $G$ with $\ch{k}(G) = \left \lceil \frac{|V(G)|}{2} \right \rceil$ be characterized? If not, can anything be said about how many edges such $G$ can be missing from being a complete graph?
\end{question}

For outerplanar graphs there is still much to be discovered. One thing to note is that when $k \geq 3$ studying outerplanar graphs becomes much more complicated. As already mentioned, we wonder how Propositon \ref{prop:outerplanar-2con-k} can be improved.

\begin{question}
    Does there exist a function $f(k)$ depending only on $k$ such that for every outerplanar graph and for every $k \geq 2$, $\ch{k}(G) \leq f(k)$?
\end{question}

In a similar direction, it would be interesting to study the $k$-hyperopic cop number of planar graphs. In particular, is it possible to find a function $f(k)$ dependant only on $k$ such that for all planar graphs $G$, $\ch{k}(G) \leq f(k)$? Some results for 0-visibility cops and robbers use treewidth and pathwidth. It would be interesting to see if one could prove similar results for the $k$-hyperopic cop number. This would in particular be useful for studying outerplanar graphs as their treewidth is at most 2.

\section*{Acknowledgements}

Vesna Ir\v{s}i\v{c} Chenoweth acknowledges the financial support from the Slovenian Research Agency (Z1-50003, P1-0297, N1-0218, N1-0285, N1-0355) and from the European Union (ERC, KARST, 101071836).

\printbibliography

\end{document}